\documentclass[a4paper,11pt,reqno]{amsart}

\usepackage[utf8]{inputenc}
\usepackage[T1]{fontenc}
\usepackage{lmodern}
\usepackage[english]{babel}
\usepackage{microtype}

\usepackage{amsmath,amssymb,amsfonts,amsthm,esint}
\usepackage{mathtools,accents}
\usepackage{mathrsfs}
\usepackage{aliascnt}
\usepackage{braket}
\usepackage{bm}

\usepackage[a4paper,margin=2.5cm]{geometry}
\usepackage[citecolor=blue, linkcolor=black, colorlinks]{hyperref}

\usepackage{enumerate}
\usepackage{xcolor}

\makeatletter
\g@addto@macro\@floatboxreset\centering
\makeatother

\makeatletter
\def\newaliasedtheorem#1[#2]#3{
	\newaliascnt{#1@alt}{#2}
	\newtheorem{#1}[#1@alt]{#3}
	\expandafter\newcommand\csname #1@altname\endcsname{#3}
}
\makeatother

\numberwithin{equation}{section}

\newtheoremstyle{slanted}{\topsep}{\topsep}{\slshape}{}{\bfseries}{.}{.5em}{}

\theoremstyle{plain}
\newtheorem{theorem}{Theorem}[section]
\newaliasedtheorem{proposition}[theorem]{Proposition}
\newaliasedtheorem{lemma}[theorem]{Lemma}
\newaliasedtheorem{corollary}[theorem]{Corollary}
\newaliasedtheorem{counterexample}[theorem]{Counterexample}

\theoremstyle{definition}
\newaliasedtheorem{definition}[theorem]{Definition}
\newaliasedtheorem{question}[theorem]{Question}
\newaliasedtheorem{openquestion}[theorem]{Open Question}
\newaliasedtheorem{conjecture}[theorem]{Conjecture}

\theoremstyle{remark}
\newaliasedtheorem{remark}[theorem]{Remark}
\newaliasedtheorem{example}[theorem]{Example}

\newcommand{\setR}{\mathbb{R}}

\newcommand{\TT}{\mathbb{T}}
\newcommand{\eps}{\varepsilon}

\let\altphi\phi
\let\phi\varphi
\let\varphi\altphi
\let\altphi\undefined

\newcommand{\norm}[1]{\left\lVert#1\right\rVert}

\newcommand{\di}{\mathop{}\!\mathrm{d}}

\newcommand{\loc}{{\rm loc}}

\newcommand{\leb}{\mathscr{L}}

\newcommand{\Prob}{\mathscr{P}}

\newfont{\tmpf}{cmsy10 scaled 2500}

\def\Xint#1{\mathchoice
	{\XXint\displaystyle\textstyle{#1}}%
	{\XXint\textstyle\scriptstyle{#1}}%
	{\XXint\scriptstyle\scriptscriptstyle{#1}}%
	{\XXint\scriptscriptstyle\scriptscriptstyle{#1}}%
	\!\int}
\def\XXint#1#2#3{{\setbox0=\hbox{$#1{#2#3}{\int}$ }
		\vcenter{\hbox{$#2#3$ }}\kern-.6\wd0}}

\def\dashint{\Xint-}

\begin{document}
	
	\title{Advection diffusion equations with Sobolev velocity field}
	\author{Elia Bru\'e}  \thanks{Scuola Normale Superiore, Piazza dei Cavalieri 7, 56126 Pisa, Italy, Email address: \url{elia.brue@sns.it, }}
	
	\author{Quoc-Hung Nguyen} \thanks{ShanghaiTech University,
		393 Middle Huaxia Road, Pudong,
		Shanghai, 201210, China, Email address: \url{qhnguyen@shanghaitech.edu.cn. }} 

	\maketitle
	
%

	\begin{abstract}
		
		In this note we study advection diffusion equations associated to incompressible $W^{1,p}$ velocity fields with $p>2$.  We present new estimates on the energy dissipation rate and we discuss applications to the study of upper bounds on the enhancing dissipation rate, lower bounds on the $L^2$ norm of the density, and quantitative vanishing viscosity estimates. The key tools employed in our argument are a propagation of regularity result, coming from the study of transport equations, and a new result connecting the energy dissipation rate to regularity estimates for transport equations.  Eventually we provide examples which underline the sharpness of our estimates.

	    \medskip
		
	\end{abstract}

\tableofcontents

\section*{Introduction and main result}\label{sec: Introduction and main results}
Let $\TT^d$ be the torus of dimension $d\ge 2$ and $T\in (0,+\infty]$. Given a divergence free velocity field $b\in L^1([0,T],W^{1,p}(\TT^d,\setR^d))$, for $p>1$, and an initial datum $u_0\in L^{\infty}(\TT^d)$ we study the Cauchy problem associated to the advection-diffusion equation 
\begin{equation}\label{CE}
  \begin{dcases}
  \partial_t u^\nu + b\cdot \nabla u^{\nu} - \nu \Delta u^{\nu} = 0
  \quad \text{on $\TT^d\times (0,T]$ }
  \\
  u^{\nu}(0,x)=u_0(x),
  \end{dcases}\tag{$E_{\nu}$}
\end{equation}
and the linear transport equation
\begin{equation}\label{Tr}
\begin{dcases}
\partial_t u^0 + b\cdot \nabla u^{0} = 0
\quad \text{on $\TT^d \times (0,T]$ }
\\
u^0(0,x)=u_0(x).
\end{dcases}\tag{$E_{0}$}
\end{equation}
Above, $\nu>0$ is a constant molecular diffusivity.
In order to ease notation we often write $u^\nu_t(x)$ and $b_t(x)$ in place of, respectively, $u^{\nu}(t,x)$ and $b(t,x)$. In this note solutions to \eqref{CE} and \eqref{Tr} are understood in the distributional sense, are mean free, and belong to the natural classes
\begin{equation}\label{eq:CEclass}
	u^{\nu}\in L^{\infty}([0,T]\times \TT^d)\cap C([0,T],L^2(\TT^d))\cap  L^1([0,T],W^{1,2}(\TT^d)),
\end{equation}
and $u^0\in C([0,T],(L^{\infty}(\TT^d),w^*))$, where $(L^{\infty}(\TT^d),w^*)$ denotes the space of bounded functions endowed with the weak-star topology.

Existence and uniqueness of solutions to \eqref{Tr} are guaranteed by the DiPerna-Lions theory \cite{lions,Ambrosio04} (see also \cite{AmbrosioCrippa14}). Regarding the advection diffusion equation, standard energy estimates ensure that \eqref{CE} posses a unique solution in \eqref{eq:CEclass} which satisfies the energy balance
\begin{equation}\label{eq:energybalance}
	\norm{u^{\nu}_t}_{L^2}^2-
	\norm{u_0}_{L^2}^2
	=-2\nu \int_0^t\norm{\nabla u_s^{\nu}}_{L^2}^2\di s
	\quad \text{for every $t\in [0,T]$.}
\end{equation}
Motivated by recent developments in the mathematical understanding of the \textit{dissipation enhancement by mixing} \cite{ConstantinKiselevRyzhikZlatos08,Bedrossian-CotiZelati17,CotiZelatiDelgadinoElgindi,FengIyer19,DrivasElgindiIyerJeong19,CotiZelatiDolce19}, in this note we study quantitative properties of solutions to \eqref{CE} at low regularity, i.e. in the setting of Sobolev divergence free velocity fields. 
This framework is quite natural in view of possible applications to problems coming from fluid dynamics and conservation laws, where very often the setting of smooth vector fields is too restrictive.

For transport problems, a theory in weaker regularity settings has been developed in the last decades and it is nowadays clear that nonuniqueness results \cite{ModenaSzekelyhidi18,ModenaSzekelyhidi19,ModenaSattig19,BrueColomboDeLellis20} and new loss of regularity phenomena \cite{AlbertiCrippaMazzuccato14,AlbertiCrippaMazzuccato16,AlbertiCrippaMazzuccato18,Jabin,BrueNguyen18c} may occur. These phenomena affect also advection diffusion problems leading to challenging open questions.

\subsection*{Enhanced dissipation and mixing}
The enhanced dissipation is the notion that solutions to \eqref{CE} dissipate the energy $\norm{u^{\nu}_t}_{L^2}$ faster than $e^{-\nu t}$, the rate at which the heat equation dissipates energy. More rigorously, we give the following definition (Cf. \cite[Definition  1]{CotiZelatiDrivas19}).

\begin{definition}\label{def:enhancingdissipation}
	Let $r:(0,\nu_0)\to (0,1)$ be an increasing function   satisfying 
	\[
	\lim_{\nu\to 0}\frac{\nu}{r(\nu)}=0.
	\]
	We say that a divergence free vector field $b$ is diffusion enhancing on a subspace $H\subset L^2(\TT^d)$, of rate $r(\nu)$, if for any $\nu\in (0,\nu_0)$ there exists $t_{\nu}>0$ such that 
	\begin{equation}
			\norm{u^{\nu}_t}_{L^2}^2\le C e^{- r(\nu) t}\norm{u_0}_{L^2}^2
			\quad\text{for every $t\ge t_\nu$, and $u_0\in H$.}
	\end{equation}
	The constant $C>0$ above depends only on $b$.
\end{definition}

It is nowadays well known that mixing in the diffusion free case is a responsible of diffusion enhancing \cite{ConstantinKiselevRyzhikZlatos08,CotiZelatiDelgadinoElgindi,FengIyer19}. 

\begin{definition}\label{def:mixing}
	Let $\rho: (0,\infty)\to [0, \infty)$ be a decreasing function satisfying $\lim_{t\to +\infty}\rho (t)=0$.
	We say that a time dependent divergence free velocity field $b$ on $\TT^d$ mixes with rate $\rho$ if for any $t_0>0$, and $u_{t_0}\in W^{1,2}$, with $\int u_{t_0}\di x=0$, denoting by $u:[t_0, \infty)\to \setR$ the solution to \eqref{CE} starting from $u_{t_0}$ at time $t=t_0$, one has
	\begin{equation*}
	   \norm{u_t}_{H^{-1}} \le \rho (t-t_0) \norm{u_{t_0}}_{W^{1,2}}
	   \quad \text{for any $t\ge t_0$.}
	\end{equation*}
\end{definition}

In \cite{CotiZelatiDelgadinoElgindi,FengIyer19} it has been estimated
the diffusion enhancing rate $r(\nu)$ in terms of the mixing rate $\rho(t)$, when the drift is Lipschitz regular uniformly in time, i.e. $b\in L^{\infty}_t W^{1,\infty}_x$.

Let us recall that, for smooth velocity fields, a simple Gronwall argument gives
\begin{equation}\label{eq:Bressanpinfty}
		\norm{u_t}_{H^{-1}} \ge e^{- t \norm{\nabla b}_{L^{\infty}}} \frac{\norm{u_0}_{\dot L^2}^2}{\norm{u_0}_{\dot H^1}}
	\quad \text{for all $t\ge 0$ and $u_0\in L^2(\TT^d)$}
\end{equation}
ensuring that the mixing rate cannot be faster than exponential.
In this meaningful case, i.e.  $\rho(t):=M e^{-\mu t}$ for some constants $M>0$ and $\mu>0$, the diffusion enhancing rate obtained in \cite[Theorem 2.5]{CotiZelatiDelgadinoElgindi} is
\begin{equation}\label{eq:diffusionenhancingratelipschitz}
	r(\nu)= C\log(1/\nu)^{-2}
	\quad
	\text{with $C=C(M, \norm{\nabla b}_{L^{\infty}})$.}
\end{equation}

As far as we know it is not known whether a velocity fields having a diffusion enhancing rate slower than $r(\nu) = O( \log(1/\nu)^{-2})$ exists. However, relying on an old result by Poon \cite{Poon96,MilesDoering18} 
\begin{equation}\label{eq:Poon}
\norm{u^{\nu}_t}_{L^2}^2\ge \norm{u_0}_{L^2}^2
\exp\left\lbrace  -\nu \frac{\norm{\nabla u_0}_{L^2}^2}{\norm{u_0}_{L^2}^2} \int_0^t 
\exp\left\lbrace 2\int_0^s \norm{\nabla b_r}_{L^{\infty}}\di r\right\rbrace
\di s\right\rbrace
\quad \forall\, t>0\ \ \nu\in (0,1)
\end{equation}
it is straightforward to see that $r(\nu)\le O( \log(1/\nu)^{-1})$ regardless of the mixing rate. 

Let us mention a very recent result of Coti Zelati and Drivas \cite{CotiZelatiDrivas19} where sharp upper bounds on the diffusion rate have been provided for a class of meaningful examples, such as shear flows and circular flows.

Out of the smooth setting it is even unknown whether a double exponential lower bound on the $L^2$ norm, as in \eqref{eq:Poon}, holds. The main difficulty here is that energy methods are not suitable to attack the problem due to a possible loss of regularity for transport equations \cite{AlbertiCrippaMazzuccato14,AlbertiCrippaMazzuccato16,AlbertiCrippaMazzuccato18,Jabin,BrueNguyen18c,BrueNguyen19}.
We refer to \cite[Section 1.3]{DrivasElgindiIyerJeong19} for a discussion on this topic.

\subsubsection*{Bressan's mixing conjecture}
Let us finally mention that in the non smooth setting is still unknown whether the mixing rate for passive scalars has a universal lower bound. This is related to the famous Bressan's mixing conjecture \cite{Bressan03} that can be formulated as follows.
\begin{conjecture}
	Given a divergence free velocity field $b\in L^{\infty}([0,\infty), W^{1,1}(\TT^d,\setR^d))$ there exist $c>0$ and $C>0$ depending only on the initial datum $u_0$ such that 
	\begin{equation*}
		\rho(t) \ge C \exp\left\lbrace - c t \norm{\nabla b}_{L^{\infty}_t L^1_x} \right\rbrace
		\quad 
		\text{for every $t\ge 0$.}
	\end{equation*}
	Where $\rho$ is the mixing rate according to \autoref{def:mixing}.
\end{conjecture}
We have already pointed out (see \eqref{eq:Bressanpinfty}) that the conjecture follows from a standard Gronwall estimate when the velocity field is Lipschitz, uniformly in time. A positive result to the conjecture has been obtained also for $b\in L^{\infty}([0,\infty), W^{1,p}(\TT^d,\setR^d))$ with $p>1$ in the very influential work \cite{CrippaDeLellis08}, while the case $p=1$ seems to require new ideas.

In view of the enhanced dissipation estimates, the problems of finding lower bounds on the energy $\norm{u_t^{\nu}}_{L^2}^2$ and on the diffusion enhancing rate $r(\nu)$ have natural connections with the challenging Bressan's mixing conjecture \cite{Bressan03}. 
We refer to \cite[section 1,3]{DrivasElgindiIyerJeong19} for a detailed discussion.

\subsection*{Energy dissipation rate in the Sobolev setting}
Aiming at better understanding enhanced dissipation and energy's lower bounds, the key quantity to study is the energy dissipation rate
\begin{equation*}
	2 \nu \int_0^t \norm{\nabla u_s^{\nu}}_{L^2}^2 \di s
	= \norm{u_0}_{L^2}^2 - \norm{u_t^{\nu}}_{L^2}^2.
\end{equation*}
Notice that, when the divergence free velocity field $b$ has the property that \eqref{Tr} admits a unique solution that conserves the $L^2$ norm, it must hold
\begin{equation}\label{zzzzz1}
	\lim_{\nu\downarrow 0}  2\nu \int_0^t \norm{\nabla u_s^{\nu}}_{L^2}^2 \di s =0.
\end{equation}
It can be easily checked by observing that, up to extracting a subsequence, $u_t^\nu \to u^0_t$ weak in $L^2$ and by using the fact that the $L^2$ norm is lower semicontinuous with respect to weak convergence.

In particular, if the drift is either Sobolev or $BV$ the DiPerna-Lions-Ambrosio theory \cite{lions,Ambrosio04,AmbrosioCrippa14} guarantees \eqref{zzzzz1} (see also the recent paper \cite{Nguyen18-1} for a quantitative analysis in $BV$ and the study  of velocity fields which can be represented as singular integral of functions in $BV$). One of the main achievement of this work is the correct estimate of the rate of convergence of \eqref{zzzzz1}. Before stating the result and its consequences let us recall that, in view of \eqref{eq:Poon} it is easily seen that in the Lipschitz setting (i.e. $b\in L^{\infty}_tW^{1,\infty}_x$) any solution to \eqref{CE} with $u_0\in W^{1,2}$ satisfies
\begin{equation}\label{eq:fastdissipation}
	\nu \int_0^1 \norm{\nabla u^{\nu}_s}_{L^2}^2\di s \le C \nu
	\quad\text{for $\nu\in (0,1)$}.
\end{equation}
Hence the energy dissipation rate is $O(\nu)$ for $\nu \to 0$. On the opposite side, if one relaxes the regularity assumption on the velocity field the situation may change dramatically. For instance, in \cite{DrivasElgindiIyerJeong19} it has been recently built a divergence free vector field 
\[
b\in C^{\infty}([0,1)\times \TT^d)\cap L^1([0,1], C^{\alpha}(\TT^d))\cap L^{\infty}([0,1]\times \TT^d)
\]
such that
\[
\limsup_{\nu\downarrow 0}\,   \nu \int_0^1 \norm{\nabla u^{\nu}_s}_{L^2}^2\di s \ge c>0,
\]
for a broad family of initial data $u_0\in W^{2,2}(\TT^d)$.
Notice that this implies the existence of passive scalars advected by $b$ with non constant $L^2$ norm.

In the Sobolev setting we have the following logarithmic rate.

\begin{theorem}\label{prop: dissipation estimate}
	Let $b\in L^{\infty}([0,T],W^{1,p}(\TT^d,\setR^d))$ be a divergence free vector field for some $p>2$.
	Any solution $u^{\nu}$ to \eqref{CE} with $u_0\in W^{1,2}(\TT^d)\cap L^{\infty}$ satisfies
	\begin{equation}\label{eq:dissipation}
	\nu\int_0^t \norm{\nabla u_s^{\nu}}^2_{L^2}\di s \le C(\norm{u_0}_{W^{1,2}}^2 +\norm{u_0}_{L^{\infty}}^2)
	\left[
	\nu t
	+\frac{t^p\norm{\nabla b}_{L^{\infty}_tL^p_x}^p+1}{\log\left( \frac{1}{\nu t}+2 \right)^{p-1}}
	\right]
	\quad \forall \, t\ge 0,
	\end{equation}
	where $C=C(p,d)$.
	In particular, for any $t>0$, we have
	\begin{equation}\label{zzzzz2}
	\nu\int_0^t \norm{\nabla u_s^{\nu}}^2_{L^2}\di s
	\le
	C
	\log(1/\nu)^{-p+1}
	\quad
	\text{for every $\nu\in (0,1/5)$}.
	\end{equation} 
	Here $C=C(p,d,t, \norm{u_0}_{W^{1,2}}^2 +\norm{u_0}_{L^{\infty}}^2 ,\norm{\nabla b}_{L^{\infty}_tL^p_x})>0$.
\end{theorem}

The next result shows that the logarithmic rate is ``almost'' sharp.

\begin{theorem}\label{th:solutionsslowdissipation}
	Let $d\ge 2$ and $p>2$ be fixed. There exist a divergence free velocity field $b\in L^{\infty}([0,1],W^{1,p}(\TT^d,\setR^d))$ and $u_0\in W^{1,2}(\TT^d)\cap L^{\infty}$ such that
	\begin{equation}\label{eq:slowdissipating}
		\limsup_{\nu\to 0}\,  \log(1/\nu)^r \, \nu \int_0^1 \norm{\nabla u^{\nu}_s}_{L^2}^2\di s=+\infty 
	\end{equation}
	for any $r> p \frac{(p-1)}{p-2}$. Here $u^{\nu}$ denotes the solution to \eqref{CE}.
\end{theorem}
We conjecture that the correct rate in \eqref{zzzzz2} is $\log(1/\nu)^{-p}$ and that \eqref{eq:slowdissipating} holds for any $r>p$. Our results come short in both ranges.

\subsubsection*{Idea of the proof of \autoref{prop: dissipation estimate}}
The crucial ingredient of proof is a new propagation of regularity result (\autoref{thm:regularity result}) for solutions to \eqref{CE}. The main novelty is that the constants appearing in the regularity estimate do not depend on the diffusivity parameter $\nu>0$. Basically it is an extension to the advection diffusion equation of a known result for transport equations \cite{BenBelgacemJabin2019,LegerFlavien16,BrueNguyen18c,BrueNguyen19}. 
We refer to \autoref{sec:reg} for a detailed outline of \autoref{thm:regularity result}. 

In order to explain the connection between propagation of regularity results and estimates on the energy dissipation rate we recall that, in the simple case $\nabla b\in L^{\infty}$, solutions to \eqref{Tr} and \eqref{CE} propagate the Sobolev regularity of the initial data for any $1\le p\le \infty$ according to
\begin{equation}\label{eq:simple}
	\norm{ \nabla u^{\nu}_t}_{L^p} \le \norm{\nabla u_0}_{L^p} e^{c t \norm{\nabla b}_{L^{\infty}}},
\end{equation}
where $c>0$ does not depend on $\nu$. This can be checked either by means of energy estimates or by studying the regularity of the stochastic flow \eqref{eq:SDE}. Having such a strong regularity result at hand the upper bound on the energy dissipation rate \eqref{eq:fastdissipation} immediately follows.

In the Sobolev setting estimates like \eqref{eq:simple} are false in general \cite{AlbertiCrippaMazzuccato14,AlbertiCrippaMazzuccato16,AlbertiCrippaMazzuccato18}. The propagated regularity is very mild \cite{LegerFlavien16,BrueNguyen18c}, and therefore not useful to  bound directly the energy dissipation rate. To overcome this problem we use an interpolation argument along with a priori estimate on $\nu^2 \int_0^t \norm{\Delta u_s^{\nu}}_{L^2}^2 \di s$ in terms of the energy dissipation rate (Cf. \autoref{prop: laplacianbound}).

\subsubsection*{Idea of the proof of \autoref{th:solutionsslowdissipation}}
To prove the existence of solutions with ``slow dissipation rate'' we exploit the existence of rough solution to the transport equation (see \autoref{prop:oldcounterexample}). 

The main idea is that quantitative bounds on the energy dissipation rate imply regularity results for transport equations. This has been made quantitative in \autoref{thm:slowly dissipating solutions} by showing the implication
\begin{equation*}
	\nu \int_0^t \norm{\nabla u_s^{\nu}}_{L^2}^2\di s \le C \log(1/\nu)^q
	\implies u_t^0\in H^{\log, r}
	\quad \text{for any $0<r< q\, \frac{p-2}{p-1},$}
\end{equation*}
when $b\in L^{\infty}_t W^{1,p}_x$. Here $H^{\log, r}$ denotes a Sobolev space of functions with ``derivative of logarithmic order'' introduced in \autoref{sec:reg}. Although the logarithmic regularity is very mild  in \cite{BrueNguyen18c} we have built solutions to \eqref{Tr}, associated to $W^{1,p}$ velocity fields, that do not propagate the $H^{\log, r}$ regularity for $r>p$. This clearly leads to the sought conclusion.

\subsubsection*{Applications:}
An immediate consequence of \autoref{th:solutionsslowdissipation} is that the double exponential lower bound as in Poon's estimate \eqref{eq:Poon} does not hold in the Sobolev setting since it forces
\begin{equation*}
	\nu \int_0^t \norm{\nabla u_s^{\nu}}_{L^2}^2 \di s \le C \nu
	\quad \text{for all $\nu\in(0,1)$.}
\end{equation*}
In view of \autoref{prop: dissipation estimate}, a natural variant of Poon's estimate is given by the following.

\begin{conjecture}\label{Conjecture:open}
	Fix $p\in [1,+\infty)$.
	Let $b\in C^{\infty}([0,T]\times \TT^d)$ be divergence free and $u_0\in W^{1,2}(\TT^d)$. Then, any solution $u_t^{\nu}$ to \eqref{CE} satisfies
	\begin{equation}\label{eq:conjectureopen}
	\norm{u_t}_{L^2}^2\ge \norm{u_0}_{L^2}^2
	\exp\left\lbrace  -\log(1/\nu)^{-p} C_1 \int_0^t 
	\exp\left\lbrace C_2\int_0^s \norm{\nabla b_r}_{L^p}\di r\right\rbrace
	\di s\right\rbrace,
	\end{equation}
	for any $\nu\in (0,1)$ and $t>0$. Here $C_1=C_1(u_0,p,d)>0$, and $C_2=C_2(p,d)>0$.
\end{conjecture}
We refer to \autoref{subsection:lowerboundonL2norm} for the discussion of a positive result (\autoref{prop:tripleexponential}) in this direction.

\medskip

An other interesting consequence of \autoref{prop: dissipation estimate} is the following upper bound on the enhanced dissipation rate in the setting of $W^{1,p}$ divergence free vector fields.

\begin{proposition}\label{prop:dissipation time scale}
	Let $b\in L^{\infty}([0,+\infty),W^{1,p}(\TT^d,\setR^d))$ be a divergence free vector field for some $p>2$. Given $u_0\in W^{1,2}(\TT^d)\cap L^{\infty}$, if there exists $r:(0,\nu_0)\to (0,+\infty)$ for some $0<\nu_0<1$, which satisfies
	\begin{equation}\label{z16}
	\norm{u^{\nu}_t}_{L^2}^2\le e^{-r(\nu)t}\norm{u_0}_{L^2}^2
	\quad\text{for any $t>1/\nu_0$ and $\nu\in (0,\nu_0)$,}	
	\end{equation}
	then
	\begin{equation}\label{z14}
	\limsup_{\nu\downarrow 0}\frac{ r(\nu)}{\log(1/\nu)^{-\frac{p-1}{p}}} <\infty.
	\end{equation}
\end{proposition}
In other words the upper bound $r(\nu) \le O(\log(1/\nu)^{-\frac{p-1}{p}})$ holds in the Sobolev setting. Notice that it is little worse than  $O(\log(1/\nu)^{-1})$, the one available for smooth vector fields.

The last application of \autoref{prop: dissipation estimate} is a quantitative estimate on the rate of convergence in the vanishing viscosity limit.

\begin{theorem}\label{th:vanishing viscosity}
	Let $b\in L^{\infty}([0,+\infty),W^{1,p}(\TT^d,\setR^d))$ be a divergence free vector field for some $p>2$. Given $u_0\in W^{1,2}(\TT^d)\cap L^{\infty}$ we consider $u^0,u^{\nu}$, respectively, solutions to \eqref{CE} and \eqref{Tr}. Then it holds
	\begin{equation}\label{Z6}
	\sup_{s\in [0,t]}\norm{u_s^\nu-u_s^0}^2_{L^2}
	\le 
	Ct\left[
	\nu+t \nu^{\frac{p-2}{p-1}}+\frac{t^{p-1}+1}{\log\left( \frac{1}{\nu t}+2 \right)^{p-2}}\right]
	\quad \text{for every $\nu>0$ and $t>0$},
	\end{equation}
	where $C=C_0(1+\norm{\nabla b}_{L^{\infty}_tL^p_x}^p)(\norm{u_0}_{W^{1/2}}^2+\norm{u_0}_{L^{\infty}}^2)$.
\end{theorem}
As far as we know \eqref{Z6} is the first quantitative vanishing viscosity estimate in terms of strong norms in the framework of Sobolev velocity fields. Previous results, such as \cite[Theorem 2]{Seis18} have dealt with weak norms.

It is worth noticing that \autoref{th:vanishing viscosity} is almost optimal, we refer to \autoref{subsection:vanishingviscositylimit} for a discussion on this.

\subsection*{Organization of the paper}

The rest of the paper is devoted to the proof of the outlined results. More specifically in \autoref{sec:reg} we present the propagation of regularity result (\autoref{thm:regularity result}) while \autoref{sec:slow dissipating solutions} is devoted to the proof of existence of ``slow dissipating solutions'' (\autoref{th:solutionsslowdissipation}). In \autoref{sec:logarithmicestimateonthedissipationrate} we show the logarithmic estimate on the energy dissipation rate (\autoref{prop: dissipation estimate}) and its corollaries. Precisely, in \autoref{subsection:lowerboundonL2norm} we present the proof of \autoref{prop:dissipation time scale} and we discuss a positive result in the direction of \autoref{Conjecture:open}. Eventually we show \autoref{th:vanishing viscosity} in \autoref{subsection:vanishingviscositylimit}.

\subsection*{Acknowledgements}
Quoc-Hung Nguyen's research was supported by the ShanghaiTech University startup fund. Part of this work was done while Quoc-Hung Nguyen was visiting Scuola Normale Superiore in Pisa.

\section{Regularity result}\label{sec:reg}
In this section, we present a propagation of regularity result for solutions to \eqref{CE}, that will play a central role in the sequel.
Here and in the rest of the paper we tacitly identify any $f:\TT^d\to \setR$ with a $1$-periodic function on $\setR^d$.

Let us begin by introducing a class of functional spaces. For any $\alpha\in (0,+\infty)$ we define 
\begin{equation}\label{eq:logSobolevsemonorm}
	[u]_{H^{\log,\alpha}}^2:= \int_{B_{1/3}}\int_{\TT^d} \frac{|u(x+h)-u(x)|^2}{|h|^d}\frac{1}{\log(1/|h|)^{1-\alpha}}\di x\di h
\end{equation}
and the related log-Sobolev class 
\begin{equation}\label{eq:logSobolevclass}
H^{\log,\alpha}:=\set{u\in L^2(\TT^d): \norm{u}_{H^{\log,\alpha}}^2:=\norm{u}_{L^2}^2+[u]_{H^{\log,\alpha}}^2<\infty}.
\end{equation}
The following characterisation of $H^{\log,\alpha}$ will play a role in the rest of the paper
\begin{equation}\label{eq:logSobolevvsFourier}
	\norm{u}_{H^{\log,\alpha}}^2\sim_{\alpha,d} \sum_{k\in \mathbb{Z}^d} \log(2+|k|)^{\alpha} |\hat u(k)|^2,
\end{equation}
where $\hat u ( k ):=\int u(x) e^{-ix\cdot k}\di x$.
We refer to \cite{BrueNguyen18c} for a proof of \eqref{eq:logSobolevvsFourier} in the case in which the ambient space is $\setR^d$.

The main result of the section is the following.

\begin{theorem}\label{thm:regularity result}
	Let $b\in L^1([0,T],W^{1,p}(\TT^d,\setR^d))$ be a divergence free vector field for some $p>1$.
	Then, any solution $u\in L^{\infty}([0,T]\times \TT^d)$ to \eqref{CE} satisfies
	\begin{align}\nonumber
	&	\int_{B_{\frac{1}{10}}} \int_{\TT^d} \frac{1\wedge | u_t(x+h)-u_t(x)|^q}{|h|^d}\frac{1}{\log(1/|h|)^{1-p}} \di x \di h\\& \lesssim_{p,q,d}\left(\int_0^t \norm{\nabla b_s}_{L^p} ds\right)^p+	\int_{B_{3/4}} \int_{\TT^d} \frac{1\wedge | u_0(x+h)-u_0(x)|^q}{|h|^d}\frac{1}{\log(1/|h|)^{1-p}} \di x \di h	\label{eq: LogSobolevContinuity} 
	\end{align} 
	for any $0< q<\infty$.\footnote{ Here we have denoted $a_1\wedge a_2$ by $\min\{a_1,a_2\}$}
\end{theorem}

In particular, choosing $q=2$ we get
\begin{corollary}\label{cor:regularity}
    Under the assumptions of \autoref{thm:regularity result} one has
	\begin{equation}\label{log-sobolev}
	[u_t]_{H^{\log,p}}\lesssim_{p,d} \left(\int_0^t\norm{\nabla b_s}_{L^p}\right)^{p/2}\norm{u_0}_{L^{\infty}}+	\norm{u_0}_{H^{\log,p}}
	\quad \text{for any $t\in [0,T]$.}
	\end{equation}
\end{corollary}
It is worth remarking that \eqref{log-sobolev} does not depend on $\nu>0$, hence the inequality holds even in the case $\nu=0$, 
i.e. for solution of the transport equation \eqref{Tr} (Cf. \cite{BrueNguyen18c,LegerFlavien16}).

Moreover, the following example borrowed from \cite[Theorem 3.2]{BrueNguyen18c}
shows that \autoref{cor:regularity} is sharp, in the sense that $H^{\log, p}$ cannot be replaced with a $H^{\log, q}$ for $q>p$.

\begin{proposition}\label{prop:oldcounterexample}
	Let $p\geq 1$. There exist a divergence free vector field $b\in L^\infty([0,+\infty); W^{1,p}(\mathbb{R}^d))$ and $ u_0\in L^\infty(\setR^d)\cap W^{1,d}(\mathbb{R}^d)$ supported, respectively, in $B_1\times
	[0,+\infty)$ and $B_1$, such that the solution $u\in L^\infty([0,+\infty)\times\setR^d)$ to \eqref{Tr} satisfies 
	\[
	u_t\notin H^{\log, q}
	\quad\text{for any $t>0$ whenever $q>p$.}
	\]
\end{proposition}

The remaining part of this section is devoted to the proof of \autoref{thm:regularity result}. The argument is a refinement of the one presented in \cite{BrueNguyen18c} and has its roots in the very influential paper \cite{CrippaDeLellis08}. In a nutshell, it goes as follows. First, by employing the Lusin-Lipschitz inequality for Sobolev maps \eqref{eq:LusinLipschitSobolev} and the Gronwall lemma, one studies regularity properties of the backwards stochastic flow (Cf. \autoref{prop:Lagrangianregularityestimate}) associated to $b$. Next, one translates the Lagrangian regularity result into an Eulerian one by using the representation formula \eqref{eq:stochastic lagrangian representation} and Lusin-type characterisations of $H^{\log, p}$ functions (Cf. \autoref{prop:loglusinliprefined}).

\begin{remark}
	In what follows it is technically convenient to assume that $b\in L^1([0,T],W^{1,p}(\TT^d,\setR^d))$ is pointwise defined, with respect to the space variable, according to
	\begin{equation}\label{eq:bpointwisedefined}
	b(t,x):=
	\begin{dcases}
	\lim_{r\downarrow 0} \frac{1}{\omega_d r^d}\int_{B_r(x)} b(t,y)\di y & \text{whener it exists,}\\
	0 & \text{otherwise.}
	\end{dcases}
	\end{equation}
\end{remark}

\subsection{Stochastic representation and Lagrangian estimate}
For any $t\in (0, \infty)$ we consider the following backward stochastic differential equation
\begin{equation}\label{eq:SDE}
\di X_{t,s} =b(s, X_{t,s})\di s+\sqrt{2\nu} \di W_s
\quad\text{with}\ \quad X_{t,t}(x)=x,
\tag{SDE}
\end{equation}
where $W_s$ is an $\TT^d$ valued Brownian motion adapted to the backwards filtration (i.e. satisfying $W_t=0$) in the probability space $(\Omega, \mathcal{F},\Prob)$.

Then, the Feynman-Kac formula \cite{Kunita97} expresses the solution of \eqref{CE} as
\begin{equation}\label{eq:stochastic lagrangian representation}
u^{\nu}(t,x)=\mathbb{E}\left[u_0\circ X_{t,0}(x)\right].
\end{equation}
Exploiting the Sobolev regularity of $b$ one gets a following Lusin type estimate for the stochastic flow map $X_{s,t}$ that does not depend on $\nu$.
\begin{proposition}\label{prop:Lagrangianregularityestimate}
	Let $b\in L^1([0,T],W^{1,p}(\TT^d,\setR^d))$ be a divergence free vector field, for some $p>1$. Fix $t\in (0,T)$. Then, there exists a nonnegative random function $g_t(\omega,x)=g_t(x)$ for $\omega\in \Omega$ and $x\in \TT^d$, which for $\Prob$-a.e. $\omega$ satisfies the inequalities
	\begin{equation}\label{zz1}
		\norm{g_t}_{L^p(\TT^d)}\lesssim_{p,d}\int_0^t \norm{\nabla b_s}_{L^p(\TT^d)}\di s,
	\end{equation}
	\begin{equation}\label{zz2}
	   e^{-g_t(x)-g_t(y)}\le 
		\frac{|X_{t,s}(x)-X_{t,s}(y)|}{|x-y|}\le e^{g_t(x)+g_t(y)}
		\quad \text{for any $0\le s\le t$, $x,y\in \TT^d$.}
	\end{equation}
	 Here $X_{t,s}$ is a version of the solution to \eqref{eq:SDE}.
\end{proposition}
\begin{proof}
	Let us introduce the local Hardy-littelwood maximal function 
	\begin{equation*}
		M f(x):=\sup_{0<r<3} \frac{1}{\omega_d r^d}\int_{B_r(x)} |f(y)|\di y,
	\end{equation*}
	for $f\in L^1(\TT^d)$, and set
	\begin{equation*}
		g_t(x):=\int_0^t M|\nabla b_s|(X_{t,s}(x))\di s
		\quad\text{for any $x\in \TT^d$,}
	\end{equation*}
	and notice that \eqref{zz1} is a simple consequence of the Minkoski inequality and the fact that $X_{t,s}$ is measure preserving.
	
	The inequality \eqref{zz2} follows from the Gronwall lemma, along with the observation that, $\Prob$-a.e., for any $x,y\in \TT^d$, the map $s\to |X_{t,s}(x)-X_{t,s}(y)|$ is absolutely continuous and satisfies
	\begin{align*}
		\frac{\di}{\di s}|X_{t,s}(x)-&X_{t,s}(y)|\le  |b(s,X_{t,s}(x))-b(s,X_{t,s}(y))|\\
		\le& C_{d}|X_{t,s}(x)-X_{t,s}(y)|\left( M|\nabla b_s|(X_{t,s}(x))+M|\nabla b_s|(X_{t,s}(y)) \right),
	\end{align*}
	for a.e. $s\in (0,t)$.
	Above we have used the the Lusin-Lipschitz inequality for Sobolev functions $f\in W^{1,1}_{\loc}(\TT^d)$, pointwise defined according to \eqref{eq:bpointwisedefined}:
	\begin{equation}\label{eq:LusinLipschitSobolev}
		|f(x)-f(y)|\le C_d|x-y|(M|\nabla f|(x)+M|\nabla f|(y))
		\quad\text{for any $x,y \in \TT^d$}.
	\end{equation}
\end{proof}

\subsection{Lusin type characterisation of $H^{\log,p}$ functions and proof of \autoref{thm:regularity result}}
Let us begin by presenting a refined version of \cite[Theorem 1.11]{BrueNguyen18c}.
\begin{proposition}\label{prop:loglusinliprefined}
	Let $q>0$ and $p>0$. For any $u\in L^1_{\loc}(\setR^d)$ it holds
	\begin{equation*}
		1\wedge | u(y)-u(x)|^q	\lesssim_{p,q,d} \log\left(1/r\right)^{-p}
		\left(G(r,x)+G(r,y)\right),
	\end{equation*}
	for any $x,y\in \setR^d$ with $2|x-y|\le r<\frac{1}{10}$,
	where
	\begin{equation*}
		G(r,z):= \int_{r\leq |h|\leq r^{1/2}} \frac{1\wedge | u(z+h)-u(z)|^q}{|h|^d}\frac{1}{\log(1/|h|)^{1-p}}\di h
		\quad\text{for any $z\in \setR^d$}.
	\end{equation*}
\end{proposition}
\begin{proof}
	First observe that, for any $x, y\in \setR^d$ and $s\geq 2|x-y|$ one has
	\begin{align}\label{eq: intermediate}
	&1\wedge |u(x)-u(y)|^q\notag\\
	&\lesssim_{d,q}
	\dashint_{B_{3s}(0)\setminus B_s(0)} 1\wedge |u(x+h)-u(x)|^q \di h+\dashint_{B_{3s}(0)\setminus B_s(0)} 1\wedge |u(y+h)-u(y)|^q \di h,
	\end{align}
    see \cite[Lemma 1.12]{BrueNguyen18c} for a simple proof.
    Next, we integrate both sides of \eqref{eq: intermediate} with respect to the variable $s$ against a suitable kernel, getting
    \begin{align*}
    1\wedge |u(x)-u(y)|^q\int_{r}^{\frac{r^{1/2}}{3}} &\frac{1}{s\log(1/s)^{1-p}} \di s\\ 
    \lesssim_{d,p}&
    \int_{r}^{\frac{r^{1/2}}{3}}  \dashint_{B_{3r}(0)\setminus B_r(0)} 1\wedge |u(x+h)-u(x)|^q \di h \frac{\di s}{s \log(1/s)^{1-p}}\\
    &+
    \int_{r}^{\frac{r^{1/2}}{3}}  \dashint_{B_{3s}(0)\setminus B_s(0)} 1\wedge |u(y+h)-u(y)|^q \di h \frac{\di s}{s \log(1/s)^{1-p}}.
    \end{align*}
    Observe that
    \begin{equation*}
    \int_{r}^{\frac{r^{1/2}}{3}}  \dashint_{B_{3s}(0)\setminus B_s(0)} 1\wedge |u(x+h)-u(x)|^q \di h \frac{\di s}{s \log(1/s)^{1-p}}
    \lesssim_{d,p} G(r,x),
    \end{equation*}
    and
    \begin{align*}
    \int_{r}^{\frac{r^{1/2}}{3}} \frac{1}{s\log(1/s)^{1-p}} \di s &
    =\frac{1}{p}\left( \log\left(\frac{1}{r}\right)^{p}-\log\left(\frac{3}{r^{1/2}}\right)^{p} \right)
    \gtrsim_{p}
    \log\left(\frac{1}{r}\right)^{p},
    \end{align*}
    where we have used $r<\frac{1}{10}$.
    The proof is complete.
\end{proof}

\begin{proof}[Proof of \autoref{thm:regularity result}]
	Let us begin by noticing that our conclusion follows from the $\Prob$-a.e. inequality
	\begin{align}\label{zz6}
		&	\int_{B_{1/5}} \int_{\TT^d} \frac{1\wedge | u_0(X_{t,0}(x+h))-u_0(X_{t,0}(x))|^q}{|h|^d}\frac{1}{\log(1/|h|)^{1-p}} \di x \di h\\& \lesssim_{p,q,d}\left(\int_0^t \norm{\nabla b_s}_{L^p} \di s\right)^p+	\int_{B_{3/4}} \int_{\TT^d} \frac{1\wedge | u_0(x+h)-u_0(x)|^q}{|h|^d}\frac{1}{\log(1/|h|)^{1-p}} \di x \di h,	\notag
	\end{align}
	by taking the expectation and using \eqref{eq:stochastic lagrangian representation}.
	
	Let us then prove \eqref{zz6}. Fix $t\in (0,T)$ and $g$ given by \autoref{prop:Lagrangianregularityestimate}, in order to keep notation short we drop the dependence of $g$ on $\omega$ and $t$. For $\Prob$-a.e. $\omega$ we have
	\begin{align*}
	\int&\int_{|h|<\frac{1}{10}}\frac{1\wedge | u_0(X_{t,0}(x+h))-u_0(X_{t,0}(x))|^q}{|h|^d\log(1/|h|)^{1-p}}\di h\di x\\
	\leq &  	\int\int_{|h|<\frac{1}{10}}\mathbf{1}_{|h|^{1/2}\exp\left\lbrace g(x+h)+g(x)\right\rbrace \geq 1}\, \frac{1}{|h|^d\log(1/|h|)^{1-p}}\di h\di x\\
	&+\int\int_{|h|<\frac{1}{10}}\mathbf{1}_{|h|^{1/2}\exp\left\lbrace g(x+h)+g(x)\right\rbrace < 1}\frac{1\wedge | u_0(X_{t,0}(x+h))-u_0(X_{t,0}(x))|^q}{|h|^d\log(1/|h|)^{1-p}}\di h\di x\\
	=:& I + II.
	\end{align*}
	Let us estimate $I$ by means of \eqref{zz1}:
	\begin{align*}
	I  \le & \int\int_{|h|<\frac{1}{10}}\mathbf{1}_{|h|^{1/2}\exp\left\lbrace g(x+h)+g(x)\right\rbrace \geq 1}\frac{1}{|h|^d\log(1/|h|)^{1-p}}\di h\di x\\
	\le & 2\int\int_{|h|<\frac{1}{10}}\mathbf{1}_{|h|^{1/2} e^{2g(x)} \geq 2}\frac{1}{|h|^d\log(1/|h|)^{1-p}}\di h\di x\\
	=& 2\int_{h<\frac{1}{10}} \leb^d\left( \set{g\ge \frac{1}{4}\log(4/|h|)}\right)\frac{1}{|h|^d\log(1/|h|)^{1-p}}\di h\\
	\lesssim_d &\int_{r<\frac{1}{10}} \leb^d\left( \set{g\ge \frac{1}{4}\log(4/r)}\right)\log(1/r)^{p-1}\frac{\di r}{r}\\
	\le & 4\int_{\log(40)/4}^{\infty}\leb^d\left( \set{g\ge \lambda}\right)(4\lambda-2\log(2))^{p-1}d\lambda\\ 
	\lesssim_{p}& \norm{g}_{L^p}^p\lesssim_{p,d} \left( \int_0^T \norm{\nabla b_s }_{L^p}\di s\right)^p.
	\end{align*}
	Let us now estimate $II$. Let $G$ be given by \autoref{prop:loglusinliprefined} and associated to $u_0$, we have
	\begin{align}\nonumber
	1\wedge & | u_0(X_{t,0}(x+h))- u_0(X_{t,0}(x))|^q\\&\lesssim \log\left (1/r\right )^{-p} (G(r, X_{t,0}(x+h))+G(r,X_{t,0}(x))) ,
	\label{zz3}
	\end{align}
	with 
	\begin{align*}
	r:= \frac{1}{20}\wedge |X_{t,0}(x+h)-X_{t,0}(x)|.
	\end{align*}
	Note that, by \autoref{prop:Lagrangianregularityestimate} we have 
	\begin{align}\label{Z1}
	\frac{1}{20}\wedge \left[ |h|\exp\left\lbrace-g(x+h)-g(x)\right\rbrace\right]\leq r \leq  \frac{1}{20}\wedge \left[ |h|\exp\left\lbrace g(x+h)+g(x)\right\rbrace\right].
	\end{align}
	Let us fix $h\in B_{\frac{1}{10}}(0)$. For any $x\in\TT^d$ such that
	$$|h|^{1/2}\exp\left\lbrace g(x+h)+g(x)\right\rbrace <1,$$
	it follows from \eqref{zz3} and \eqref{Z1} that 
	$|h|^{3/2}\leq  r \leq |h|^{1/2},$
	and 
    \begin{align*}
    1\wedge & | u_0(X_{t,0}(x+h))- u_0(X_{t,0}(x))|^q\\&\lesssim \log\left (\frac{1}{|h|}\right )^{-p} (H(|h|, X_{t,0}(x+h))+H(|h|,X_{t,0}(x))),
    \end{align*}
    where 
    \begin{equation*}
    H(r,z):= \int_{ r^{3/2}\leq |h|\leq  r^{1/4}} \frac{1\wedge | u_0(z+h)-u_0(z)|^q}{|h|^d}\frac{1}{\log(1/|h|)^{1-p}}\di h
    \quad\text{for any $z\in \TT^d$}.
    \end{equation*}
	This implies, 
	\begin{align*}
	II\lesssim_{p} &  \int\int_{|h|<\frac{1}{10}}\frac{H(|h|, X_{t,0}(x+h))}{|h|^d\log(1/|h|)}\di h\di x
	+\int\int_{|h|<\frac{1}{10}}\frac{H(|h|, X_{t,0}(x))}{|h|^d\log(1/|h|)}\di h\di x
	\\=& 2\int\int_{|h|<\frac{1}{10}}\frac{H(|h|, X_{t,0}(x))}{|h|^d\log(1/|h|)}\di h\di x
	\\	\simeq_{p,d} & \int\int_0^{\frac{1}{10}}\frac{1}{r\log(1/r)} \int_{r^{3/2}\leq |h|\leq  r^{1/4}} \frac{1\wedge | u_0(x+h)-u_0(x)|^q}{|h|^d\log(1/|h|)^{1-p}}\di h\di r\di x
	\\	\lesssim_{p,d} & \int\int_{|h|<3/4}\left(\int_{|h|^4}^{|h|^{2/3}}\frac{1}{r\log(1/r)} \di r\right) \frac{1\wedge | u_0(x+h)-u_0(x)|^q}{|h|^d\log(1/|h|)^{1-p}}\di h\di x
	\\	\simeq_{p,d} & \int\int_{|h|<3/4} \frac{1\wedge | u_0(x+h)-u_0(x)|^q}{|h|^d}\frac{1}{\log(1/|h|)^{1-p}}\di h\di x,
	\end{align*} 
	here we have used the fact that 
	\begin{equation*}
	\int_{|h|^4}^{|h|^{2/3}}\frac{1}{r\log(1/r)} \di r =\log(\log(1/|h|^4))-\log(\log(1/|h|^{2/3}))=\log(6).
	\end{equation*}
	The proof is over.
\end{proof}

\begin{remark}
	Notice that \eqref{zz6} is stronger than the regularity estimate in \eqref{thm:regularity result}, indeed when we take the expectation we are losing information.
	We believe that a more precise analysis, which do not lose this information,  could lead to the following improved version of \eqref{cor:regularity}:
	\begin{equation}\label{zzzz1}
	[u_t]_{H^{\log,p}}^2 + \nu \int_0^t[\nabla u_s]_{H^{\log,p}}^2 \di s
	\lesssim_{d,p}
	 \left(\int_0^t\norm{\nabla b_s}_{L^p}\right)^{p}\norm{u_0}_{L^{\infty}}^2+	\norm{u_0}_{H^{\log,p}}^2
	\quad \text{$\forall t\in [0,T]$.}
	\end{equation}
	Unfortunately we are not able to show this estimate by means of our approach. However it is worth stressing that if \eqref{zzzzz1} were true then it would lead to significant improvements of \autoref{prop: dissipation estimate}, \autoref{th:solutionsslowdissipation} and their applications. 
\end{remark}

\section{Proof of \autoref{th:solutionsslowdissipation}: existence of slow dissipating solutions}
\label{sec:slow dissipating solutions}

The core of the argument in the proof of \autoref{th:solutionsslowdissipation} is the following.

\begin{proposition}\label{thm:slowly dissipating solutions}
	Let $b\in L^1([0,T],W^{1,p}(\TT^d,\setR^d))$ be a divergence free vector field, for some $p>2$. Let $u^{\nu}$ and $u^0$ solve, respectively, \eqref{CE} and \eqref{Tr}. For any $t\in [0,T]$,
	if there exists $q>0$ such that
	\begin{equation}\label{zzz9}
	\limsup_{\nu\downarrow 0} 
	\,  \log(1/\nu)^q \,
	\nu \int_0^t \norm{\nabla u^{\nu}_s}_{L^2}^2\di s <\infty,
	\end{equation}
	then $u_t\in H^{\log ,r}$ (see \eqref{eq:logSobolevsemonorm}) for any $0 < r < q \frac{p-2}{p-1}$.
\end{proposition}

\begin{remark}
	By exploiting the ideas developed in the proof of \autoref{thm:slowly dissipating solutions} (Cf. \autoref{remark:Sobolev}) one can prove the following variant:
	if there exists $\theta\in (0,1]$ such that
	\begin{equation}\label{zzz10}
	\limsup_{\nu\downarrow 0} 
	\, 
	\nu^{1-\theta} \int_0^t \norm{\nabla u^{\nu}_s}_{L^2}^2\di s <\infty,
	\end{equation}
	then $u_t\in H^{r}(\TT^d)$ for any $0 < r < \theta \frac{p-2}{2(p-1)}$.
	Here $H^{r}(\TT^d):=\set{u\in L^2(\TT^d): [u]_{H^{r}}<\infty}$ denotes the fractional Sobolev space defined by means of the Gagliardo semi-norm
	\begin{equation*}
		[u]_{H^{r}}^2:=\int_{B_2} \int_{\TT^d} \frac{|u(x+h)-u(x)|^2}{|h|^{d + 2r}}\di x \di h.
	\end{equation*}
\end{remark}

\begin{proof}[Proof of \autoref{th:solutionsslowdissipation} given \autoref{thm:slowly dissipating solutions}]

 We argue by contradiction. If the conclusion were false then the assumptions of \autoref{thm:slowly dissipating solutions} are satisfied for some $q>p\frac{p-1}{p-2}$, therefore there exists $r>p$ such that $u^0_t\in H^{\log, r}$. This is not possible in general in view of \autoref{prop:oldcounterexample}.
\end{proof}

\subsection{Interpolation estimate} 
In this subsection we present an estimate on $\nu^2\int_0^t \norm{\Delta u^\nu_s}_{L^2}^2\di s$, which plays a central role in \autoref{thm:slowly dissipating solutions} and \autoref{prop: dissipation estimate}. 
\begin{proposition}\label{prop: laplacianbound}
	Let $\gamma\in (2,+\infty]$ be fixed. Assume $b\in L^{\infty}([0,T],W^{1,p}(\TT^d,\setR^d))$ for some $p>\frac{2\gamma}{\gamma-2}$.
	Any solution $u^{\nu}\in L^{\infty}([0,T], W^{1,2}(\TT^d)\cap L^{\gamma})$ to \eqref{CE} satisfies
	\begin{align}\label{z1}
	\nu\norm{\nabla u^{\nu}_t}_{L^2}^2+ & \nu^2\int_0^t \norm{\Delta u^\nu_s}_{L^2}^2\di s \notag\\	
	\le& \nu \norm{\nabla u_0}_{L^2}^2+C_{d,p,\gamma}\norm{u_0}_{L^{\gamma}}^{2(1-\beta)}\norm{\nabla b}_{L^\infty_tL^p_x}^{2-\beta}t^{1-\beta}\left(\nu\int_0^t \norm{\nabla u^{\nu}_s}^2_{L^2}\di s\right)^\beta,
	\end{align}
	where
	\begin{equation*}
		\beta=1-\frac{1}{p-1-\frac{2p}{\gamma}}\in (0,1).
	\end{equation*}
\end{proposition}

In the sequel we will use \eqref{z1} just in the case $\gamma=\infty$.
\begin{proof}
It is enough to prove the result for $\norm{\nabla b}_{L^\infty_tL^p_x}=1$, the general case follows by a simple scaling argument. 
Testing \eqref{CE} against $\Delta u_t$ we get
\begin{align*}
\norm{\nabla u^{\nu}_t}_{L^2}^2+2\nu\int_0^t \norm{\Delta u^\nu_s}_{L^2}^2\di s 
&\le \norm{\nabla u_0}_{L^2}^2+\int_0^t\norm{\nabla u^\nu_s}_{L^{2p'}}^2\di s\\
 &\le \norm{\nabla u_0}_{L^2}^2 +\int_0^t\norm{\nabla u^\nu_s}_{L^{2}}^{2\alpha}\norm{\nabla u^\nu_s}_{L^{2q}}^{2(1-\alpha)}\di s,
\end{align*}
with 
\begin{equation}\label{zz4}p'=\frac{p}{p-1},~~
\frac{1}{p'}=\alpha+\frac{1-\alpha}{q},\quad \alpha\in (0,1).
\end{equation}
By using the Gagliardo–Nirenberg interpolation inequality we deduce
\begin{equation}\label{zz5}
\norm{\nabla u^{\nu}}_{L^{2q}}
\le C_{d,q,\gamma} \norm{\Delta u^{\nu}}_{L^2}^{1/2}\norm{u^{\nu}}_{L^\gamma}^{1/2}
\quad \text{for}\quad \frac{1}{q}=\frac{1}{2}+\frac{1}{\gamma},
\end{equation}
hence
\begin{align*}
\norm{\nabla u^{\nu}_t}_{L^2}^2+2\nu\int_0^t \norm{\Delta u^\nu_s}_{L^2}^2\di s 
&\le \norm{\nabla u_0}_{L^2}^2+C_{d,q,\gamma}\norm{u_0}_{L^{\gamma}}^{1-\alpha}\int_0^t\norm{\nabla u^{\nu}_s}_{L^2}^{2\alpha}\norm{\Delta u^{\nu}_s}_{L^2}^{1-\alpha}\di s\\
&\le \norm{\nabla u_0}_{L^2}^2+C_{d,q,\gamma,\alpha}\norm{u_0}_{L^{\gamma}}^{2\frac{1-\alpha}{1+\alpha}}\nu^{-\frac{1-\alpha}{1+\alpha}}\int_0^t\norm{\nabla u^{\nu}_s}_{L^2}^{\frac{4\alpha}{1+\alpha}}\di s\\&+\nu \int_0^t \norm{\Delta u^\nu_s}_{L^2}^2\di s,
\end{align*}
that amounts to
\begin{align*}
\norm{\nabla u^{\nu}_t}_{L^2}^2+\nu\int_0^t \norm{\Delta u^\nu_s}_{L^2}^2\di s 	
&\le \norm{\nabla u_0}_{L^2}^2+C_{d,q,\gamma,\alpha}\norm{u_0}_{L^{\gamma}}^{2\frac{1-\alpha}{1+\alpha}}\nu^{-\frac{1-\alpha}{1+\alpha}}\int_0^t\norm{\nabla u^{\nu}_s}_{L^2}^{\frac{4\alpha}{1+\alpha}}\di s\\
&\le \norm{\nabla u_0}_{L^2}^2+C_{d,q,\gamma,\alpha}\norm{u_0}_{L^{\gamma}}^{2\frac{1-\alpha}{1+\alpha}} t\nu^{-1}\left( t^{-1}\nu\int_0^t \norm{\nabla u_s^{\nu}}^2_{L^2}\di s\right)^{\frac{2\alpha}{1+\alpha}}.
\end{align*}
In order to conclude the proof we just need to combine \eqref{zz4} and \eqref{zz5} to find the expression of $\alpha$ and $q$ in terms of $p$ and $\gamma$.
\end{proof}

\subsection{Proof of \autoref{thm:slowly dissipating solutions}}
\label{subsection:proofofslowlydissipating}
Fix $t\in (0,T)$ and a convolution kernel $\rho_\eps(x):=\eps^{-d}\rho(x\eps^{-1})$ where $\rho\in C^{\infty}_c(\setR^d)$ and $\eps>0$. For any $f\in L^1(\TT^d)$ we denote by
\begin{equation*}
	f\ast\rho_\eps(x):=\int_{\setR^d} u(x-y)\rho_{\eps}(y)\di y
\end{equation*}
its convolution against $\rho_{\eps}$, which is continuous and $1$-periodic.
Then, for any $\nu>0$, it holds
\begin{align}\nonumber
	\norm{u^0_t\ast\rho_{\eps}-u^0_t}_{L^2}
	\le & \norm{u^0_t\ast\rho_{\eps}-u^{\nu}_t\ast\rho_{\eps}}_{L^2}
	+\norm{u_t^{\nu}\ast\rho_{\eps}-u_t^{\nu}}_{L^2}+\norm{u_t^{\nu}-u_t^0}_{L^2}\\
	\le & 2\norm{u_t^{\nu}-u_t^0}_{L^2}+\norm{u_t^{\nu}\ast\rho_{\eps}-u_t^{\nu}}_{L^2}\notag\\
	\lesssim &  \norm{u_t^{\nu}-u_t^0}_{L^2}+\eps\norm{\nabla u^{\nu}_t}_{L^2}.
	\label{Z11}
\end{align}
From \eqref{zzz5} and \autoref{prop: laplacianbound} (with $\gamma=\infty$) we get
\begin{equation}\label{zzz6}
\norm{u_t^\nu-u_t^0}^2_{L^2}
\lesssim_p  t\nu \norm{\nabla u_0}_{L^2}^2 +t^{\frac{p}{p-1}}
\norm{u_0}_{L^2}^{\frac{2}{p-1}}
\norm{\nabla b}_{L^\infty_tL^p_x}^{\frac{p}{p-1}}
\left(\nu\int_0^t \norm{\nabla u_s^{\nu}}^2_{L^2}\di s\right)^{\frac{p-2}{p-1}},
\end{equation}
while \autoref{prop: laplacianbound} and \eqref{eq:energybalance} yield
\begin{equation}\label{zzz2}
\eps^2 \norm{\nabla u_t^{\nu}}_{L^2}^2\lesssim_{p}
 \eps^2 \norm{\nabla u_0}_{L^2}^2 + \eps^2 \nu^{-1} t^{\frac{1}{p-1}}
 \norm{u_0}_{L^2}^{\frac{2}{p-1}}
 \norm{\nabla b}_{L^\infty_tL^p_x}^{\frac{p}{p-1}}
 \left(\nu\int_0^t \norm{\nabla u_s^{\nu}}^2_{L^2}\di s\right)^{\frac{p-2}{p-1}}.
\end{equation}
By combining \eqref{Z11}, \eqref{zzz2}, \eqref{zzz6}, assuming without loss of generality $\norm{u_0}_{W^{1,2}}+\norm{u_0}_{L^{\infty}}\le 1$, and choosing $\eps=\nu$ one gets
\begin{equation}\label{zzz7}
		\norm{u^0_t\ast\rho_{\nu}-u^0_t}_{L^2}^2
		\lesssim_p 
		 \nu ( t + 1 ) + t^{\frac{p}{p-1}}
		\norm{\nabla b}_{L^\infty_tL^p_x}^{\frac{p}{p-1}}
		\left(\nu\int_0^t \norm{\nabla u_s^{\nu}}^2_{L^2}\di s\right)^{\frac{p-2}{p-1}}
		\quad\text{for every $\nu\in (0,1)$}.
\end{equation}
Thanks to \eqref{zzz9} there exists $\nu_0\in (0,1)$ such that  $\nu \int_0^t\norm{\nabla u_s}_{L^2}^2\di s\le C\log(1/\nu)^{-q}$ for any $\nu\in (0,\nu_0)$, hence
\begin{equation}\label{zzz4}
\norm{u^0_t\ast\rho_{\nu}-u^0_t}_{L^2}^2\lesssim_{t,p} \log(1/\nu) ^{-q\frac{p-2}{p-1}}
\quad\text{for any $0<\nu<\nu_0$}.
\end{equation}
We claim that \eqref{zzz4} implies $u^0_t\in H^{\log,r}$ for every $0<r < q\frac{p-2}{p-1}$.
To this end we note that
\begin{align*}
\sum_{k\in\mathbb{Z}} | \widehat{u^0_t} ( k ) |^2  \int_0^{\nu_0} \frac{ |\hat\rho( \nu k )-1 |^2 }{\log( 1/\nu )^{1-r}}  \frac{\di \nu}{\nu}
=
\int_0^{\nu_0}\int_{\TT^d} | u^0_t \ast \rho_{\nu} - u^0_t |^2\frac{1}{\log(1/\nu)^{1-r}} \di x \frac{\di \nu}{\nu}
\lesssim \frac{1}{ q\frac{p-2}{p-1}-r }
\end{align*}
for any $0< r<q\frac{p-2}{p-1}$, where $\hat{\rho}$ denotes the Fourier transform of $\rho$ in $\setR^d$. Moreover it is not hard to check that 
\begin{equation}
C_{\nu_0}+\int_{0}^{\nu_0}|\hat\rho(\nu k )-1|^2\frac{1}{\log(1/\nu)^{1-r}} \frac{d\nu}{\nu}\gtrsim_{\nu_0,d} \log(2+| k |)^{r},
\end{equation}
Thus \eqref{eq:logSobolevvsFourier} yields
\begin{align*}
\norm{u^0_t}_{H^{\log,r}}^2\lesssim_{r,d}
 \sum_{k\in \mathbb{Z}^d} \log(2+|k|)^r |\widehat{u^0_t}(k)|^2 <\infty.
\end{align*}
The proof is over.

\begin{remark}\label{remark:Sobolev}
	Under the assumption \eqref{zzz10}, the estimate \eqref{zzz7} gives
	\begin{equation*}
	\norm{u^0_t\ast\rho_{\nu}-u^0_t}_{L^2}^2\lesssim_{t,p} \nu ^{\theta\frac{p-2}{p-1}}
	\quad\text{for any $0<\nu<\nu_0$},
	\end{equation*}
	for some $\nu_0>0$. Hence $u^0_t\in H^{r}(\mathbb{T}^d)$ for any $0<r<\theta\frac{p-2}{2(p-1)}$.
\end{remark}

%

\section{Logarithmic estimate on the dissipation rate and consequences}
\label{sec:logarithmicestimateonthedissipationrate}
In this section we prove \autoref{prop: dissipation estimate} and we draw a series of consequences.

\subsection{Proof of \autoref{prop: dissipation estimate}: logarithmic bound on the dissipation rate}
    Since \eqref{CE} is linear we can assume without loss of generality that
    \begin{equation*}
    	\norm{u_0}_{W^{1,2}}+\norm{u_0}_{L^2}\le 1.
    \end{equation*}
    Observe that, from \eqref{eq:logSobolevvsFourier}, we deduce
    \begin{equation*}
      [u_0]_{H^{\log,p}}^2\lesssim_{d,p} \sum_{k\in \mathbb{Z}^d}\log(2+|k|)^p|\hat u_0(k)|^2
      \lesssim \sum_{k\in \mathbb{Z}^d} (1+|k|^2)|\hat u_0(k)|^2
      \le \norm{u_0}_{W^{1,2}}^2  \le 1.
    \end{equation*}
    We apply \eqref{prop: laplacianbound} with $\gamma=\infty$ and $\beta=\frac{p-2}{p-1}$ obtaining
	\begin{align}\nonumber
	\nu \norm{\nabla u_t}_{L^2}^2+&\nu^2\int_0^t \norm{\Delta u^\nu_s}_{L^2}^2\di s\\ 
	\le& \nu \norm{\nabla u_0}_{L^2}^2+C_p\norm{u_0}_{L^{\infty}}^{\frac{2}{p-1}}\norm{\nabla b}_{L^\infty_tL^p_x}^{\frac{p}{p-1}}t^{\frac{1}{p-1}}\left(\nu\int_0^t \norm{\nabla u_s^{\nu}}^2_{L^2}\di s\right)^{\frac{p-2}{p-1}}\nonumber\\
	\le & \nu +C_{p,d}\norm{\nabla b}_{L^\infty_tL^p_x}^{\frac{p}{p-1}}t^{\frac{1}{p-1}}\left(\nu\int_0^t \norm{\nabla u_s^{\nu}}^2_{L^2}\di s\right)^{\frac{p-2}{p-1}}. 	\label{z12}	
	\end{align}
	Let us now set
	\begin{equation}
	D_{\nu}(t):=\nu\int_0^t \norm{\nabla u_s^{\nu}}^2_{L^2}\di s,
	\end{equation}
	and fix  $\lambda>0$. Exploiting \eqref{eq:logSobolevvsFourier} and \eqref{z12} we obtain
	\begin{align*}
	D_{\nu}(t) & 
	=\nu \int_0^t \sum_{k\in \mathbb{Z}^d} |k|^2|\widehat{u^\nu_t}(k)|^2\\
	&\lesssim  \frac{\nu\lambda^2}{\log(\lambda+2)^p}\int_0^t\sum_{|k|<\lambda} \log (2+|k|)^p|\widehat {u^\nu_s}(k)|^2 \di s +\frac{\nu}{\lambda^2} \int_0^t \norm{\Delta u^\nu_s}_{L^2}^2\di s\\
	&\le \frac{\nu\lambda^2 }{\log(\lambda+2)^p}\int_{0}^{t}\norm{u_s}_{H^{\log,p}}^2\di s
	+\frac{1}{\lambda^2}+\frac{1}{\nu\lambda^2}\norm{\nabla b}_{L^\infty_tL^p_x}^{\frac{p}{p-1}}t^{\frac{1}{p-1}}D_{\nu}(t)^{\frac{p-2}{p-1}}
	\end{align*}
	for any $t\in (0,T)$.
	By means of the Young inequality we can estimate
	\begin{equation*}
	\frac{1}{\nu\lambda^2}\norm{\nabla b}_{L^\infty_tL^p_x}^{\frac{p}{p-1}}t^{\frac{1}{p-1}}D_{\nu}(t)^{\frac{p-2}{p-1}}
	\le \frac{C_p}{\nu^{p-1}\lambda^{2(p-1)}}\norm{\nabla b}_{L^\infty_tL^p_x}^{p}t
	+\frac{1}{2} D_{\nu}(t),
	\end{equation*}
	while \autoref{cor:regularity} gives
	\begin{equation*}
	\int_{0}^{t}\norm{u_s}_{H^{\log,p}}^2\di s 
	\lesssim_{p,d} t^{p+1}\norm{\nabla b}_{L^{\infty}_tL^p_x}^p+t.
	\end{equation*}
	Putting all together we end up with 
	\begin{align}
	\label{z13}
	D_{\nu}(t) \lesssim_{p,d}  \frac{\nu\lambda^2t}{\log(\lambda+2)^p}\left(t^p\norm{\nabla b}_{L^{\infty}_tL^p_x}^p+1\right)
	+\frac{1}{\lambda^2} 
	+\frac{1}{\nu^{p-1}\lambda^{2(p-1)}}t\norm{\nabla b}_{L^\infty_tL^p_x}^{p}.
	\end{align}
	Choosing
	\begin{equation}
	\lambda=(\nu t)^{-\frac{1}{2}} \log\left(\frac{1}{\nu t}+e\right)^{1/2},
	\end{equation}
	and using the elementary inequality
	\begin{equation*}
	\frac{\log\left( \frac{1}{\nu t}+e \right)}{\log\left( \frac{\log((\nu t)^{-1}+e)^{1/2}}{\sqrt{\nu t}}+2 \right)^p}
	\le
	\frac{\log\left( \frac{1}{\nu t}+e \right) }{\log\left( \frac{1}{\sqrt{\nu t}}+2 \right)^p}
	\le 2^p\frac{1}{\log\left( \frac{1}{\nu t}+4 \right)^{p-1}},
	\end{equation*}
	one gets \eqref{eq:dissipation}.    

\subsection{Lower bound on $L^2$ norms}
\label{subsection:lowerboundonL2norm}
Let us now present two consequences of \autoref{prop: dissipation estimate}.

The first conclusion is \autoref{prop:dissipation time scale} below. It provides an upper bound on the enhanced diffusion rate $r(\nu)$, we refer to the introduction of this not for a detailed discussion.

\begin{proposition}
	Let $b\in L^{\infty}([0,+\infty),W^{1,p}(\TT^d,\setR^d))$ be a divergence free vector field for some $p>2$. Given $u_0\in W^{1,2}(\TT^d)\cap L^{\infty}$, if there exists $r:(0,\nu_0)\to (0,+\infty)$ for some $0<\nu_0<1$, which satisfies
	\begin{equation}
	\norm{u^{\nu}_t}_{L^2}^2\le e^{-r(\nu)t}\norm{u_0}_{L^2}^2
	\quad\text{for any $t>1/\nu_0$ and $\nu\in (0,\nu_0)$,}	
	\end{equation}
	then
	\begin{equation}
	\limsup_{\nu\downarrow 0}\frac{ r(\nu) }{\log(1/\nu)^{-\frac{p-1}{p}}}\le C,
	\end{equation}
	where $C=C(p,d,\norm{u_0}_{W^{1,2}},\norm{u_0}_{L^{\infty}}, \norm{\nabla b}_{L^{\infty}_tL^p_x})$.
\end{proposition}

\begin{proof}
	Let $C=C(p,d)$ as in the statement of \autoref{prop: dissipation estimate} and set
	\begin{equation*}
	K:= C(\norm{u_0}_{W^{1,2}}^2 +\norm{u_0}_{L^{\infty}}^2).
	\end{equation*}
	Fix $\alpha>0$ to be chosen later. Given $\nu>0$ small enough we set $t:=\alpha\log(1/\nu)^{\frac{p-1}{p}}>1/\nu_0$. From \eqref{eq:energybalance} and \autoref{prop: dissipation estimate} we get
	\begin{align*}
	\norm{u_t}_{L^2}^2=& \norm{u_0}_{L^2}^2-\nu \int_0^t \norm{\nabla u_s}_{L^2}^2\di s\\
	\ge& \norm{u_0}_{L^2}^2\left[ 1- \frac{K}{\norm{u_0}_{L^2}^2}\left(\nu t+ \frac{t^p\norm{\nabla b}_{L^{\infty}_tL^p_x}^p+1}{ \log\left( \frac{1}{\nu t}+2 \right)^{p-1}}\right) \right]\\
	=& \norm{u_0}_{L^2}^2\left[ 1-\frac{K}{\norm{u_0}_{L^2}^2}
	\left( \alpha\nu\log(1/\nu)^{\frac{p-1}{p}}
	+ 
	\alpha^p \frac{\log(1/\nu)^{p-1} \norm{\nabla b}_{L^{\infty}_tL^p_x}^p +1}{\log\left( \frac{1}{\alpha \nu }\log(1/\nu)^{-\frac{p-1}{p}}+2 \right)^{p-1}}
	\right)
	\right]\\
	=& \norm{u_0}_{L^2}^2\left(1-\frac{K \alpha^p \norm{\nabla b}_{L^{\infty}_tL^p_x}^p }{\norm{u_0}_{L^2}^2} + o(1)\right) 
	\end{align*}
	where $o(1)\to 0$ for $\nu\to 0$, and $K$ is as in \autoref{prop: dissipation estimate}. We deduce
	\begin{equation*}
	\liminf_{\nu\downarrow 0}\,  \exp\left\lbrace-\alpha\frac{ r(\nu)}{\log(1/\nu)^{-\frac{p-1}{p}}}\right\rbrace 
	\ge 1-\frac{K \alpha^p \norm{\nabla b}_{L^{\infty}_tL^p_x}^p }{\norm{u_0}_{L^2}^2},
	\end{equation*}
	and choosing $\alpha$ such that 
	\[\frac{K \alpha^p \norm{\nabla b}_{L^{\infty}_tL^p_x}^p }{\norm{u_0}_{L^2}^2}=\frac{1}{2},
	\]
	 we easily get \eqref{z14}.
\end{proof}

A second consequence of \autoref{prop: dissipation estimate} is a step toward \autoref{Conjecture:open}.

\begin{proposition}\label{prop:tripleexponential}
	Let $b\in L^{\infty}([0,+\infty),W^{1,p}(\TT^d,\setR^d)\cap L^{\infty})$ be a divergence free vector field for some $p>2$. Let $u^{\nu}$ solve \eqref{CE} with $u_0\in W^{1,2}(\TT^d)\cap L^{\infty}$. Then, for any $\alpha\in [0,p-1)$ there exist $\nu_0=\nu_0(u_0, \norm{b}_{L^\infty_tW^{1,p}_x\cap L^\infty},\alpha, p,d)\in (0,1)$ and $C=C(u_0, \norm{b}_{L^\infty_tW^{1,p}_x\cap L^\infty},p,d)>0$ such that
	\begin{equation}\label{eq:tripleexp}
	\norm{u^{\nu}_t}_{L^2}^2
	\ge \norm{ u_0 }_{ L^2 }^2  \exp\left\lbrace-\log(1/\nu)^{-\alpha } \exp\left\lbrace   e^{C t^{\frac{p}{p-1-\alpha}}}  \right\rbrace \right\rbrace,
	\end{equation}
	for every $1<t<+\infty$ and $\nu\in (0,\nu_0)$.	
\end{proposition}

\begin{proof}
	Let $C=C(p,d)$ as in the statement of \autoref{prop: dissipation estimate} and define
	\begin{equation*}
	K:= C(\norm{u_0}_{W^{1,2}}^2 +\norm{u_0}_{L^{\infty}}^2).
	\end{equation*}
Set 
\[
t_{\nu}:=\left(\frac{\norm{u_0}_{L^2}^2}{2 K}\right)^{1/p}\frac{1}{\norm{\nabla b}_{L^{\infty}_tL^p_x}} \log(1/\nu)^{\frac{p-1-\alpha}{p}}.
\]
Let us begin by considering the case $0<t\le t_{\nu}$, arguing as in the proof of \autoref{prop:dissipation time scale}, we get
\begin{align*}
\norm{u_t}_{L^2}^2
=& \norm{u_0}_{L^2}^2-\nu \int_0^t \norm{\nabla u_s}_{L^2}^2\di s
\ge \norm{u_0}_{L^2}^2-\nu \int_0^{t_{\nu}}\norm{\nabla u_s}_{L^2}^2\di s\\
\ge& \norm{u_0}_{L^2}^2\left[ 1- \frac{K}{\norm{u_0}_{L^2}^2}\left(\nu t_{\nu}+ \frac{t_{\nu}^p\norm{\nabla b}_{L^{\infty}_tL^p_x}^p+1}{ \log\left( \frac{1}{\nu t_{\nu}}+2 \right)^{p-1}}\right) \right]\\
=& \norm{u_0}_{L^2}^2\left(1-\frac{1}{2}\log(1/\nu)^{-\alpha} + o(1)\right) 
\end{align*}
where $o(1)\to 0$ for $\nu\to 0$. Therefore, we can find
$\nu_0=\nu_0(u_0,b,p,d)\in (0,1)$ such that, for any $\nu \in (0,\nu_0)$ it holds
\begin{equation}\label{z17}
\norm{u^{\nu}_t}_{L^2}^2\ge e^{-\log(1/\nu)^{-\alpha}} \norm{u_0}_{L^2}^2
\quad \text{for every $t\in (0,t_{\nu})$}.
\end{equation}
Observe that \eqref{z17} implies \eqref{eq:tripleexp} for any $t\in (0, t_{\nu})$.

Let us now consider the case $t>t_{\nu}$, for $\nu\in (0,\nu_0)$. From \cite{MilesDoering18} we know that
\begin{equation}
\norm{u^\nu_t}_{L^2}^2\geq 	\norm{u_0}_{L^2}^2\exp\left\lbrace-\nu^2\frac{\norm{\nabla u_0}_{L^2}^2}{\norm{u_0}_{L^2}^2}\left(e^{t\nu^{-1}\norm{b}_{L^\infty}}-1\right)\right\rbrace,
\end{equation}
it is easily seen that, for  $t\ge t_{\nu}$ one has
\begin{equation}
\nu^2\frac{\norm{\nabla u_0}_{L^2}^2}{\norm{u_0}_{L^2}^2}\left(e^{t\nu^{-1}\norm{b}_{L^\infty}}-1\right)
\le 
\frac{\norm{\nabla u_0}_{L^2}^2}{\norm{u_0}_{L^2}^2}
\exp\left\lbrace \norm{b}_{L^{\infty}} e^{Ct^{\frac{p}{p-1-\alpha}}}\right\rbrace,
\end{equation}
where $C=C(u_0,\norm{b}_{L^\infty_tW^{1,p}_x\cap L^\infty},p,d)>0$,
hence \eqref{eq:tripleexp} is satisfied provided $\nu_0>0$ is small enough. 
\end{proof}

\subsection{Vanishing viscosity limit}
    \label{subsection:vanishingviscositylimit}
    Another interesting consequence of \autoref{prop: dissipation estimate} regards the vanishing viscosity limit $\nu\to 0$. More precisely we aim at estimating the $L^2$ distance between $u^{\nu}$ and $u^0$ which, respectively, solve \eqref{CE} and \eqref{Tr}. To this end the key estimate to take into account is
    \begin{equation}\label{eq:laplacian estimate}
	\nu^2 \int_0^t \norm{\Delta u_s}_{L^2}^2\di s
	\le 
	C\left[
	\nu+t \nu^{\frac{p-2}{p-1}}+\frac{t^{p-1}+1}{\log\left( \frac{1}{\nu t}+2 \right)^{p-2}}\right]
	\quad \text{for every $\nu>0$ and $t>0$},
    \end{equation}
    where $C=(1+\norm{\nabla b}_{L^{\infty}_tL^p_x}^p)(\norm{u_0}_{W^{1/2}}^2+\norm{u_0}_{L^{\infty}}^2)$. Notice that \eqref{eq:laplacian estimate} easily follows by combining \eqref{prop: laplacianbound} and \eqref{prop: dissipation estimate}.

    The connection between \eqref{eq:laplacian estimate} and the vanishing viscosity estimate is given by
    \begin{equation}\label{zzz5}
    \sup_{s\in [0,t]}\norm{u_s^\nu-u_s^0}^2_{L^2}\le  t\nu^2 \int_0^t \norm{\Delta u^{\nu}_s}^2\di s,
    \end{equation}
    that comes from
	\begin{align*}
	\frac{\di}{\di t}\norm{u_t^\nu-u_t^0}^2_{L^2}\leq 2\nu\left|\int(u_t^\nu-u_t^0)\Delta u_t^\nu \di x\right|
	\le 2\nu \norm{u^{\nu}_t-u_0}_{L^2}\norm{\Delta u^{\nu}_t}_{L^2},
	\end{align*}
    by applying the H\"older inequality.
What we have proven is \autoref{th:vanishing viscosity} that we state again below for the reader's convenience.

\begin{theorem}
	Let $b\in L^{\infty}([0,+\infty),W^{1,p}(\TT^d,\setR^d))$ be a divergence free vector field for some $p>2$. Given $u_0\in W^{1,2}(\TT^d)\cap L^{\infty}$ we consider $u^0,u^{\nu}$, respectively, solutions to \eqref{CE} and \eqref{Tr}. Then it holds
	\begin{equation*}
    \sup_{s\in [0,t]}\norm{u_s^\nu-u_s^0}^2_{L^2}
    \le 
    Ct\left[
    \nu+t \nu^{\frac{p-2}{p-1}}+\frac{t^{p-1}+1}{\log\left( \frac{1}{\nu t}+2 \right)^{p-2}}\right]
    \quad \text{for every $\nu>0$ and $t>0$},
    \end{equation*}
    where $C=(1+\norm{\nabla b}_{L^{\infty}_tL^p_x}^p)(\norm{u_0}_{W^{1/2}}^2+\norm{u_0}_{L^{\infty}}^2)$.
\end{theorem}

Relying on ideas developed in \autoref{subsection:proofofslowlydissipating} we are able to prove 
that the bound 
\[
 \sup_{s\in [0,t]}\norm{u_s^\nu-u_s^0}^2_{L^2} \le O( \log(1/\nu)^{2-p})
 \quad\text{for $\nu\to 0$}
\] 
is almost optimal. More precisely, we show that for any $C>0$ one can find $b\in L^1([0,T],W^{1,p}(\TT^d,\setR^d))$ and $u_0\in W^{1,2}(\setR^d)\cap L^{\infty}$ such that, for every $r>p$ it holds 
\begin{equation*}
\limsup_{\nu\downarrow 0}\, \log(1/\nu)^r \norm{u_t^{\nu}-u^0_t}_{L^2}^2=\infty.
\end{equation*}
This easily follows from \autoref{prop:counterexample} below and the example in \autoref{prop:oldcounterexample}.

\begin{proposition}\label{prop:counterexample}
	Fix $u_0\in W^{1,2}(\TT^d)\cap L^{\infty}$.
	Let $u_t^{\nu}$ and $u_t^0$ solve, respectively, \eqref{CE} \eqref{Tr} with $b\in L^1([0,T],W^{1,p}(\TT^d,\setR^d))$, for some $p>2$.
	If there exist $t\in (0,T)$, $\nu_0\in (0,1)$, $C>0$ and $r>0$ such that 
	\begin{equation}\label{Z9}
	\norm{u^\nu_t-u^0_t}^2_{L^2} \le C\log(1/\nu)^{-r},
	\quad
	\text{for every } 0<\nu <\nu_0,
	\end{equation}
	then
	\begin{equation*}
     u_t^{0}\in H^{\log, r_1}
	\quad \text{for any $0<r_1<r$}.
	\end{equation*}
\end{proposition}

\begin{proof}
	We can assume without loss of generality that 
	$$\norm{\nabla b}_{L^{\infty}_tL^{p}_x}+\norm{u_0}_{W^{1,2}}+\norm{u_0}_{L^{\infty}}\le 1.
	$$
	Fix $\nu\in (0,\nu_0)$ and $\eps\in (0,1)$. By \eqref{Z11} and our assumptions we have
	\begin{align*}
	\norm{u^0_t\ast\rho_{\eps}-u^0_t}_{L^2}
	&\le 2\norm{u_t^{\nu}-u_t^0}_{L^2}+\norm{u_t^{\nu}\ast\rho_{\eps}-u_t^{\nu}}_{L^2}\\
	&\lesssim  C\log(1/\nu)^{-r/2}+
	\eps \norm{\nabla u_t^{\nu}}_{L^2},
	\end{align*}
	that along with \autoref{prop: laplacianbound}, gives
	\begin{equation}\label{Z15}
		\norm{u^0_t\ast\rho_{\eps}-u^0_t}_{L^2}
		\lesssim_{p,d,\gamma} C\log(1/\nu)^{-r/2}+\eps\nu^{-1/2} t^{(1-\beta)/2}.
	\end{equation}
	In particular, choosing $\eps=\nu$, there exists $C'=C'(t,C,p,d,\gamma)$ such that 
	\begin{equation}\label{zzz8}
	\norm{u^0_t\ast\rho_{\nu}-u^0_t}_{L^2} 
	\le C' \log(1/\nu)^{-r/2}
	\quad \text{for every }0<\nu<\nu_0.
	\end{equation}
	As we have already shown in \autoref{subsection:proofofslowlydissipating}, the inequality \eqref{zzz8} implies $u^{\nu}_t\in H^{\log,r_1}$ for any $0\le r_1<r$.
\end{proof}

\end{document}